\documentclass[a4paper]{article}
\usepackage{mathrsfs, mathtools, amsthm, amssymb, thm-restate}  
\usepackage{paralist, tablists}  

\usepackage{graphicx, xcolor, tikz}  
\usetikzlibrary{calc, shapes, decorations.pathreplacing, decorations.markings}  
\usepackage{caption}  
\usepackage[labelformat=simple]{subcaption}  

\usepackage[left=25mm, top=25mm, bottom=25mm, right=25mm]{geometry}  
\usepackage[T1]{fontenc} 

\usepackage[inline]{enumitem}  
\usepackage[square, numbers, sort&compress]{natbib}  
\usepackage[
  bookmarks=true,                
  bookmarksnumbered=true,        
  bookmarksopen=true,            
  plainpages=false,              
  colorlinks=true,               
  linkcolor=blue,                
  citecolor=black,               
  anchorcolor=green,             
  urlcolor=blue,                 
  hyperindex=true                
]{hyperref} 

\newtheorem{theorem}{Theorem} 
\newtheorem{corollary}{Corollary}
\theoremstyle{definition}
\newtheorem{definition}{Definition}
\newtheorem{remark}{Remark}
\newtheorem{proposition}{Proposition}

\makeatletter
\def\th@plain{%
  \upshape 
}
\makeatother

\newcommand{\etal}{et~al.\ }

\usepackage{marvosym,wasysym}
\usepackage{array}
\usepackage{bm}  

\makeatletter
\renewenvironment{proof}[1][\proofname]{\par
  \pushQED{\qed}%
  \normalfont \topsep6\p@\@plus6\p@\relax
  \trivlist
  \item[\hskip\labelsep
        \bfseries
    #1\@addpunct{.}]\ignorespaces
}{%
  \popQED\endtrivlist\@endpefalse
}
\makeatother

\usepackage[capitalise]{cleveref}  

\begin{document}

\title{Variable degeneracy on toroidal graphs}
\author{Rui Li\footnote{School of Mathematics and Statistics, Henan University, Kaifeng, 475004, P. R. China} \and Tao Wang\footnote{Center for Applied Mathematics, Henan University, Kaifeng, 475004, P. R. China. {\tt Corresponding
author: wangtao@henu.edu.cn; https://orcid.org/0000-0001-9732-1617} } }

\date{}
\maketitle
\begin{abstract}
DP-coloring was introduced by Dvo{\v{r}}{\'{a}}k and Postle as a generalization of list coloring and signed coloring. A new coloring, strictly $f$-degenerate transversal, is a further generalization of DP-coloring and $L$-forested-coloring. In this paper, we present some structural results on planar and toroidal graphs with forbidden configurations, and establish some sufficient conditions for the existence of strictly $f$-degenerate transversal based on these structural results. Consequently, (i) every toroidal graph without subgraphs isomorphic to the configurations in \cref{NOA} is DP-$4$-colorable, and has list vertex arboricity at most $2$, (ii) every toroidal graph without $4$-cycles is DP-$4$-colorable, and has list vertex arboricity at most $2$, (iii) every planar graph without subgraphs isomorphic to the configurations in \cref{E} is DP-$4$-colorable, and has list vertex arboricity at most $2$. These results improve upon previous results on DP-$4$-coloring [Discrete Math. 341~(7) (2018) 1983--1986; Bull. Malays. Math. Sci. Soc. 43~(3) (2020) 2271--2285] and (list) vertex arboricity [Discrete Math. 333 (2014) 101--105; Int. J. Math. Stat. 16~(1) (2015) 97--105; Iranian Math. Soc. 42~(5) (2016) 1293--1303]. 


\end{abstract}

\section{Introduction}
All graphs considered in this paper are finite, undirected and simple. An \textbf{$L$-forested-coloring} of $G$ for a list assignment $L$ is a coloring $\phi$ (not necessarily proper) such that $\phi(v) \in L(v)$ for each $v \in V(G)$ and each color class induces a forest. The \textbf{list vertex arboricity} of a graph $G$ is the least integer $k$ such that $G$ has an $L$-forested-coloring for any list $k$-assignment $L$. The list vertex arboricity of a graph is the list version of vertex arboricity. 

Let $\mathbb{Z}^{*}$ denote the set of nonnegative integers, and let $f$ be a function mapping $V(G)$ to $\mathbb{Z}^{*}$. A graph $G$ is said to be \textbf{strictly $f$-degenerate} if every nonempty subgraph $\Gamma$ of $G$ contains a vertex $v$ such that $\deg_{\Gamma}(v) < f(v)$. A \textbf{cover} of a graph $G$ is a graph $H$ with vertex set $V(H) = \bigcup_{v \in V(G)} L_{v}$, where $L_{v} = \{\,(v, 1), (v, 2), \dots, (v, s)\,\}$; the edge set $\mathscr{M} = \bigcup_{uv \in E(G)}\mathscr{M}_{uv}$, where $\mathscr{M}_{uv}$ is a matching between $L_{u}$ and $L_{v}$. Note that $L_{v}$ is an independent set in $H$ and $\mathscr{M}_{uv}$ may be an empty set. This definition is slightly different from Bernshteyn and Kostochka's \cite{MR3803061}, but consistent with Schweser's \cite{MR4020554}. A vertex subset $R \subseteq V(H)$ is a \textbf{transversal} of $H$ if $|R \cap L_{v}| = 1$ for each $v \in V(G)$. 

Let $H$ be a cover of $G$ and $f$ be a function mapping $V(H)$ to $\mathbb{Z}^{*}$. We refer to the pair $(H, f)$ as a \textbf{valued cover} of $G$. Let $S$ be a subset of $V(G)$, and $H_{S}$ denote the induced subgraph $H[\bigcup_{v \in S} L_{v}]$. A transversal $R$ is a \textbf{strictly $f$-degenerate transversal} if $H[R]$ is strictly $f$-degenerate.  

Let $H$ be a cover of $G$, and let $f$ be a function mapping $V(H)$ to $\{0, 1\}$. An \textbf{independent transversal}, also known as a \textbf{DP-coloring}, of $H$ is a strictly $f$-degenerate transversal of $H$. It is observed that a DP-coloring is a special type of independent set in $H$. The concept of DP-coloring, also known as \textbf{correspondence coloring}, was introduced by Dvo{\v{r}}{\'{a}}k and Postle \cite{MR3758240}. Furthermore, it has been shown \cite{MR4357325} that the concept of strictly $f$-degenerate transversal generalizes several other coloring concepts, such as list coloring, $(f_{1}, f_{2}, \dots, f_{s})$-partition, signed coloring, DP-coloring, and $L$-forested-coloring. For some related results on this topic, we refer the reader to \cite{MR4345150,MR4114324,MR4401835}. 

The \textbf{DP-chromatic number $\chi_{\mathrm{DP}}(G)$} of a graph $G$ is the least integer $k$ such that $(H, f)$ has a DP-coloring whenever $H$ is a cover of $G$ and $f$ is a function mapping $V(H)$ to $\{0, 1\}$ with $f(v, 1) + f(v, 2) + \dots + f(v, s) \geq k$ for each $v \in V(G)$. A graph $G$ is \textbf{DP-$k$-colorable} if its DP-chromatic number is at most $k$. 

Dvo{\v{r}}{\'{a}}k and Postle \cite{MR3758240} introduced a non-trivial application of DP-coloring to solve a longstanding conjecture by Borodin \cite{MR3004485}, showing that every planar graph without cycles of lengths $4$ to $8$ is $3$-choosable. The other application of DP-coloring can be found in \cite{MR3803061}, Bernshteyn and Kostochka extended Dirac's theorem on the minimum number of edges in critical graphs to Dirac's theorem on the minimum number of edges in DP-critical graphs, yielding a solution to the problem posed by Kostochka and Stiebitz \cite{MR1883593}. 

%

Dvo{\v{r}}{\'{a}}k and Postle \cite{MR3758240} showed that every planar graph is DP-$5$-colorable, and observed that $\chi_{\mathrm{DP}}(G) \leq k + 1$ if $G$ is $k$-degenerate. Thomassen \cite{MR1290638} showed that every planar graph is $5$-choosable, and Voigt \cite{MR1235909} showed that there are planar graphs that are not $4$-choosable. Thus it is interesting to give sufficient conditions for planar graphs to be $4$-choosable. Additionally, as a generalization of list coloring, it is also interesting to give sufficient conditions for planar graphs to be DP-$4$-colorable. Kim and Ozeki \cite{MR3802151} showed that for each $k \in \{3, 4, 5, 6\}$, every planar graph without $k$-cycles is DP-$4$-colorable. Two cycles are \textbf{adjacent} if they share at least one edge. Kim and Yu \cite{MR3969022} showed that every planar graph without triangles adjacent to $4$-cycles is DP-$4$-colorable. Sittitrai and Nakprasit \cite{MR4089638} showed that every planar graph without pairwise adjacent $3$-, $4$-, and $5$-cycles is DP-$4$-colorable. For  additional results on DP-colorings, we refer the reader to \cite{MR3679840, MR3518419, MR3889157, MR4043754, MR3957361}. 

Raspaud and Wang \cite{MR2408378} showed that every planar graph without $4$-cycles has vertex arboricity at most $2$. A graph is \textbf{toroidal} if it can be embedded in the torus. Since any graph that can be embedded in the plane can be embedded in the torus, every planar graph is also a toroidal graph. Choi and Zhang \cite{MR3233411} showed that every toroidal graph without $4$-cycles has vertex arboricity at most $2$, which improves the result in \cite{MR2408378}. Zhang \cite{MR3570576} showed that every toroidal graph without $5$-cycles has list vertex arboricity at most $2$. Huang \etal \cite{MR3320048} showed that every toroidal graph without $3$-cycles adjacent to $5$-cycles has list vertex arboricity at most $2$. Chen \etal \cite{MR3508765} showed that every toroidal graph without $3$-cycles adjacent to $4$-cycles has list vertex arboricity at most $2$. 

Consider a graph $G$ and a valued cover $(H, f)$ of $G$. The pair $(H, f)$ is said to be \textbf{minimal non-strictly $f$-degenerate} if $H$ does not have a strictly $f$-degenerate transversal, but $(H - L_{v}, f)$ has a strictly $f$-degenerate transversal for any $v \in V(G)$. A \textbf{$k$-cap} is a chordless $k$-cycle together with a new vertex that is adjacent to precisely two adjacent vertices on the cycle. A \textbf{$k$-cap-subgraph} $F_{k}$ of a graph $G$ is a subgraph isomorphic to the $k$-cap, where all vertices having degree $4$ in $G$, see an illustration of $F_{5}$ in \cref{F35}. 

\begin{figure}
    \centering
    \begin{tikzpicture}
        \coordinate (A) at (72:1);
        \coordinate (B) at (144:1);
        \coordinate (C) at (216:1);
        \coordinate (D) at (288:1);
        \coordinate (E) at (0:1);
        \coordinate (B') at (180:1.15);

        \draw (A) node[right]{\small$4$} -- (B) node[above]{\small$4$} -- (C) node[below]{\small$4$} -- (D) node[right]{\small$4$} -- (E) node[right]{\small$4$} -- cycle;
        \draw (B) -- (B') node[left]{$4$} -- (C);

        \foreach \pt in {A, B, C, D, E, B'} {
            \node[rectangle, inner sep = 2, fill, draw] at (\pt) {};
        }
    \end{tikzpicture}
    \caption{The $5$-cap-subgraph.}
    \label{F35}
\end{figure}
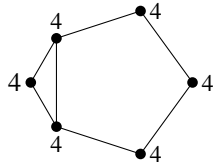

In \cref{sec:ML}, we show the following structural result on certain toroidal graphs. 

\begin{figure}
    \centering
    \subcaptionbox{\label{fig:subfig:a-}}{%
        \begin{tikzpicture}
        \coordinate (A) at (45:1);
        \coordinate (B) at (135:1);
        \coordinate (C) at (225:1);
        \coordinate (D) at (-45:1);
        \coordinate (H) at (90:1.414);
        \draw (A)--(H)--(B)--(C)--(D)--cycle;
        \draw (A)--(B);
        \foreach \pt in {A, B, C, D, H} {
            \node[circle, inner sep = 1.5, fill = white, draw] at (\pt) {};
        }
        \end{tikzpicture}
    }\hspace{1.5cm}%
    \subcaptionbox{\label{fig:subfig:b-}}{%
        \begin{tikzpicture}
        \coordinate (A) at (30:1);
        \coordinate (B) at (150:1);
        \coordinate (C) at (225:1);
        \coordinate (D) at (-45:1);
        \coordinate (H) at (90:1.414);
        \coordinate (X) at (60:1.4);
        \coordinate (Y) at (120:1.4);
        \draw (A)--(X)--(H)--(Y)--(B)--(C)--(D)--cycle;
        \draw (A)--(H)--(B);
        \foreach \pt in {A, B, C, D, H, X, Y} {
            \node[circle, inner sep = 1.5, fill = white, draw] at (\pt) {};
        }
\end{tikzpicture}}
\caption{Forbidden configurations in \cref{MLONE,MRONEPLANAR,MRONE}.}
\label{NOA}
\end{figure}
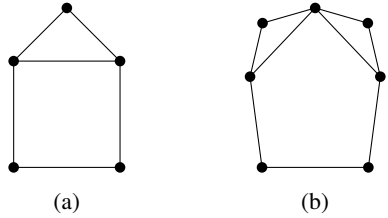

\begin{restatable}{theorem}{MLONE}\label{MLONE}
A connected toroidal graph that does not include subgraphs isomorphic to the configurations as depicted in \cref{NOA}, has minimum degree at most $3$, unless it is a $2$-connected $4$-regular graph with Euler characteristic $\epsilon(G) = 0$. 
\end{restatable}

Building upon this structural result, along with other notable structural results on critical graphs (specifically, those concerning minimal non-strictly $f$-degenerate graphs), we can easily obtain the following theorem for a certain class of planar graphs. 

\begin{restatable}{theorem}{MRONEPLANAR}\label{MRONEPLANAR}
Let $G$ be a planar graph without subgraphs isomorphic to the configurations in \cref{NOA}, and let $(H, f)$ be a valued cover of $G$. If $f(v, 1) + f(v, 2) + \dots + f(v, s) \geq 4$ for each $v \in V(G)$, then $H$ has a strictly $f$-degenerate transversal. 
\end{restatable}

For toroidal graphs without subgraphs isomorphic to the configurations in \cref{NOA}, we have the following result. 
\begin{restatable}{theorem}{MRONE}\label{MRONE}
Let $G$ be a toroidal graph without subgraphs isomorphic to the configurations in \cref{NOA}, and let $(H, f)$ be a valued cover of $G$. If $f(v, 1) + f(v, 2) + \dots + f(v, s) \geq 4$ for each $v \in V(G)$, and $(H_{S}, f)$ is not a monoblock whenever $G[S]$ is a 2-connected $4$-regular graph, then $H$ has a strictly $f$-degenerate transversal. 
\end{restatable}


\begin{figure}%
\centering
\subcaptionbox{\label{fig:subfig:a--}}
{\begin{tikzpicture}
\coordinate (A) at (45:1);
\coordinate (B) at (135:1);
\coordinate (C) at (225:1);
\coordinate (D) at (-45:1);
\coordinate (H) at (90:1.414);
\draw (A)--(H)--(B)--(C)--(D)--cycle;
\draw (A)--(B);
        \foreach \pt in {A, B, C, D, H} {
            \node[circle, inner sep = 1.5, fill = white, draw] at (\pt) {};
        }
\end{tikzpicture}}\hspace{1.5cm}
\subcaptionbox{\label{fig:subfig:b--}}
{\begin{tikzpicture}
\coordinate (A) at (30:1);
\coordinate (B) at (150:1);
\coordinate (C) at (225:1);
\coordinate (D) at (-45:1);
\coordinate (H) at (90:1.414);
\coordinate (X) at (60:1.4);
\coordinate (Y) at (120:1.4);
\coordinate (T) at ($(H)!(A)!(X)$);
\coordinate (Z) at ($(A)!2!(T)$); 
\draw (A)--(X)--(Z)--(H)--(Y)--(B)--(C)--(D)--cycle;
\draw (A)--(H)--(B);
\draw (H)--(X);
        \foreach \pt in {A, B, C, D, H, X, Y, Z} {
            \node[circle, inner sep = 1.5, fill = white, draw] at (\pt) {};
        }
\end{tikzpicture}}\hspace{1.5cm}
\subcaptionbox{\label{fig:subfig:c--}}
{\begin{tikzpicture}
\coordinate (A) at (30:1);
\coordinate (B) at (150:1);
\coordinate (C) at (225:1);
\coordinate (D) at (-45:1);
\coordinate (H) at (90:1.414);
\coordinate (X) at (60:1.4);
\coordinate (Y) at (120:1.4);
\coordinate (T) at ($(A)!(H)!(X)$);
\coordinate (Z) at ($(H)!2!(T)$); 
\draw (A)--(Z)--(X)--(H)--(Y)--(B)--(C)--(D)--cycle;
\draw (X)--(A)--(H)--(B);
        \foreach \pt in {A, B, C, D, H, X, Y, Z} {
            \node[circle, inner sep = 1.5, fill = white, draw] at (\pt) {};
        }
\end{tikzpicture}}
\caption{Forbidden configurations in \cref{MLTHREE,MRTHREE}.}
\label{E}
\end{figure}
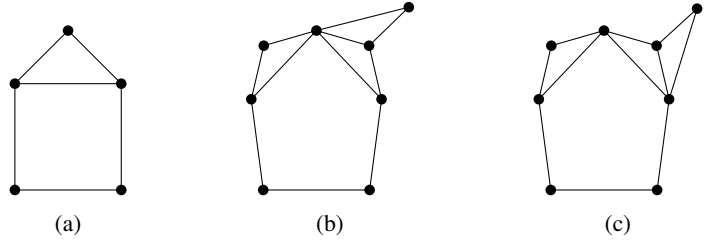

The subsequent structural conclusion can be derived from the proof of Theorem 1.9 in \cite{MR3233411}. For the sake of completeness and readability, we present a self-contained proof in \cref{sec:ML}. 
\begin{restatable}[Choi and Zhang \cite{MR3233411}]
{theorem}{MLTWO}\label{MLTWO}If $G$ is a connected toroidal graph without $4$-cycles, then either $G$ has minimum degree at most $3$, or $G$ contains a $k$-cap-subgraph for some $k \geq 5$. 
\end{restatable}

Using this structural result, we formulate the subsequent theorem for toroidal graphs without $4$-cycles. 

\begin{restatable}{theorem}{MRTWO}\label{MRTWO}
Let $G$ be a toroidal graph without $4$-cycles. Let $H$ be a cover of $G$ and $f$ be a function mapping $V(H)$ to $\{0, 1, 2\}$. If $f(v, 1) + f(v, 2) + \dots + f(v, s) \geq 4$ for all $v \in V(G)$, then $H$ has a strictly $f$-degenerate transversal. 
\end{restatable}

In \cref{sec:NA}, we show that every planar graph without subgraphs isomorphic to the configurations in \cref{E} has minimum degree at most $3$ unless it contains a $5$-cap-subgraph. 

\begin{restatable}{theorem}{MLTHREE}\label{MLTHREE}
If $G$ is a planar graph without subgraphs isomorphic to the configurations in \cref{E}, then it has minimum degree at most $3$ or it contains a $5$-cap-subgraph.  
\end{restatable}

Similarly, we can easily deduce the subsequent result from \cref{MLTHREE}. 
\begin{restatable}{theorem}{MRTHREE}\label{MRTHREE}
Let $G$ be a planar graph without subgraphs isomorphic to the configurations in \cref{E}. Let $H$ be a cover of $G$ and $f$ be a function mapping $V(H)$ to $\{0, 1, 2\}$. If $f(v, 1) + f(v, 2) + \dots + f(v, s) \geq 4$ for each $v \in V(G)$, then $H$ has a strictly $f$-degenerate transversal. 
\end{restatable}

Recall that the notion of strictly $f$-degenerate transversal serves as a generalization of DP-coloring and $L$-forested-coloring.

\begin{proposition}[Lu \etal \cite{MR4357325}]\label{PROP}
Let $G$ be a graph and $H$ be a cover of $G$. If $H$ has a strictly $f$-degenerate transversal whenever $f$ is a function mapping $V(H)$ to $\{0, 1, 2\}$, and $f(v, 1) + f(v, 2) + \dots + f(v, s) \geq k$ for each vertex $v \in V(G)$, then $G$ is DP-$k$-colorable, and the list vertex arboricity of $G$ is at most $\lceil\frac{k}{2}\rceil$. 
\end{proposition}

Using \cref{MRONE,MRTWO,MRTHREE} and \cref{PROP}, we obtain the following corollary. 

\begin{corollary}
\mbox{}
\begin{enumerate}[label = (\roman*)]
\item Every toroidal graph without subgraphs isomorphic to the configurations in \cref{NOA}, is DP-$4$-colorable. 
\item The list vertex arboricity of any toroidal graph without subgraphs isomorphic to the configurations in \cref{NOA} is at most $2$. 
\item Every toroidal graph without $4$-cycles is DP-$4$-colorable. 
\item The list vertex arboricity of any toroidal graph without $4$-cycles is at most $2$. 
\item Every planar graph without subgraphs isomorphic to the configurations in \cref{E}, is DP-$4$-colorable. 
\item The list vertex arboricity of any planar graph without subgraphs isomorphic to the configurations in \cref{E} is at most $2$. 
\end{enumerate}
\end{corollary}

The results in the above corollary enhance many known results concerning DP-$4$-coloring and (list) vertex arboricity in \cite{MR3570576,MR3320048,MR3233411,MR2408378,MR4089638}. 

\section{Preliminaries}\label{sec:minimal}
We define three classes of graphs as follows.

\begin{itemize}
\item The graph $\widetilde{K_{p}}$ is the Cartesian product of the complete graph $K_{p}$ and an independent $s$-set. 

\item The \textbf{circular ladder graph $\Gamma_{n}$} is the Cartesian product of the cycle $C_{n}$ and an independent set with two vertices. 

\item The \textbf{M\"{o}bius ladder} $M_{n}$ is the graph with the vertex set $\big\{\,(i, j) \mid i \in [n], j \in [2]\,\big\}$, where two vertices $(i, j)$ and $(i', j')$ are adjacent if and only if either
\begin{enumerate}[label = ---]
\item $i' = i + 1$ and $j = j'$ for $1 \leq i \leq n-1$, or 
\item $i = n$, $i'=1$ and $j \neq j'$. 
\end{enumerate}
\end{itemize}

We introduce the following definitions:

\begin{definition}
Given a valued cover $(H, f)$ of a graph $G$, a \textbf{kernel} of $H$ is the subgraph obtained from $H$ by removing each vertex $(u, j)$ with $f(u, j) = 0$. 
\end{definition}

\begin{definition}
A \textbf{building cover} for a $2$-connected graph $B$ is a valued cover $(H, f)$ of $B$ such that 
\[
f(v, 1) + f(v, 2) + \dots + f(v, s) = \deg_{B}(v)
\] for each $v \in V(B)$, and one of the following conditions is satisfied:
\begin{enumerate}[label = (\roman*)]
\item The kernel of $H$ is isomorphic to $B$. Such a cover is referred to as a \textbf{monoblock}. 
\item If $B$ is isomorphic to a complete graph $K_{p}$ for some $p \geq 2$, then the kernel of $H$ is isomorphic to $\widetilde{K_{p}}$, where $f$ is constant on each component of $\widetilde{K_{p}}$. 
\item If $B$ is isomorphic to an odd cycle, then the kernel of $H$ is isomorphic to the circular ladder graph with $f \equiv 1$. 
\item If $B$ is isomorphic to an even cycle, then the kernel of $H$ is isomorphic to the M\"{o}bius ladder with $f \equiv 1$. \qed
\end{enumerate}
\end{definition}

\begin{definition}\label{constructible}
Every building cover is \textbf{constructible}. A valued cover $(H, f)$ of a graph $G$ is also \textbf{constructible} if it is derived from a constructible valued cover $(H^{(1)}, f^{(1)})$ of $G^{(1)}$ and a constructible valued cover $(H^{(2)}, f^{(2)})$ of $G^{(2)}$ such that all of the following conditions hold: 
\begin{enumerate}[label = (\roman*)]
\item the graph $G$ is obtained from $G^{(1)}$ and $G^{(2)}$ by merging $w_{1}$ in $G^{(1)}$ and $w_{2}$ in $G^{(2)}$ into a new vertex $w$, and
\item the cover $H$ is obtained from $H^{(1)}$ and $H^{(2)}$ by merging $(w_{1}, q)$ and $(w_{2}, q)$ into a new vertex $(w, q)$ for each $q \in [s]$, and  
\item $f(w, q) = f^{(1)}(w_{1}, q) + f^{(2)}(w_{2}, q)$ for each $q \in [s]$, $f = f^{(1)}$ on $H^{(1)} - L_{w}$, and $f = f^{(2)}$ on $H^{(2)} - L_{w}$. We simply write this as $\bm{f = f^{(1)} + f^{(2)}}$. \qed
\end{enumerate}
\end{definition}

\begin{theorem}[Lu \etal \cite{MR4357325}]\label{MR}
Let $G$ be a connected graph and $(H, f)$ be a valued cover with $f(v, 1) + f(v, 2) + \dots + f(v, s) \geq \deg_{G}(v)$ for every vertex $v \in V(G)$. Then $H$ has a strictly $f$-degenerate transversal if and only if $H$ is non-constructible.
\end{theorem}

Let $\mathscr{D}$ denote the set of all vertices $v \in V(G)$ for which the condition $f(v, 1) + f(v, 2) + \dots + f(v, s) \geq \deg_{G}(v)$ holds.

\begin{theorem}[Lu \etal \cite{MR4357325}]\label{L}
Let $G$ be a graph and $(H, f)$ be a valued cover of $G$. Suppose $B$ is a nonempty subset of $\mathscr{D}$, and $G[B]$ has no cut vertices. If $(H, f)$ is a minimal non-strictly $f$-degenerate pair, then
\begin{enumerate}[label = (\roman*)]
\item\label{M1} $G$ is connected and $f(v, 1) + f(v, 2) + \dots + f(v, s) \leq \deg_{G}(v)$ for each vertex $v \in V(G)$, and 
\item\label{M2} Either $G[B]$ is a cycle, or a complete graph, or $\deg_{G[B]}(v) \leq \max_{q} \{f(v, q)\}$ for each vertex $v \in B$. \qed
\end{enumerate}
\end{theorem}

\section{Results on toroidal graphs}\label{sec:ML}
We recall our structural result on some toroidal graphs. 
\MLONE*
\begin{proof}
Suppose that $G$ is a connected toroidal graph without the configurations in \cref{NOA}, and the minimum degree is at least $4$. We can further assume that $G$ has been 2-cell embedded in the plane or torus. 

We then assign the initial charge $\mu(v) = \deg(v) - 4$ to each vertex $v \in V(G)$, and $\mu(f) = \deg(f) - 4$ to each face $f \in F(G)$. According to Euler's formula, the sum of these initial charges is $- 4\epsilon(G)$. Then

\begin{equation}\label{Euler1}
\sum_{v\,\in\,V(G)}\big(\deg(v) - 4\big) + \sum_{f\,\in\,F(G)}\big(\deg(f) - 4\big) = - 4\epsilon(G) \leq 0. 
\end{equation}

We use the following discharging rules. 
\begin{enumerate}[label = \textbf{R\arabic*.}, ref = R\arabic*]
\item\label{Rule1} A $3$-face adjacent to three $5^{+}$-faces receives $\frac{1}{3}$ from each adjacent face. 
\item\label{Rule2} Let $w_{1}w_{2}$ be incident with two $3$-faces $f = [w_{1}w_{2}w_{3}]$ and $g = [w_{1}w_{2}w_{4}]$. If $\deg(w_{1}) = \deg(w_{2}) = 4$, then $f$ receives $\frac{1}{2}$ from each adjacent $5^{+}$-face; otherwise, $f$ receives $\frac{1}{3}$ from each adjacent $5^{+}$-face and $\frac{1}{3}$ from each $5^{+}$-vertex in $\{w_{1}, w_{2}\}$. 
\end{enumerate}

Since each $5$-cycle has no chords, each $3$-face can be adjacent to at most one $3$-face and at least two $5^{+}$-faces. When a $3$-face $f$ is adjacent to three $5^{+}$-faces, it receives $\frac{1}{3}$ from each adjacent $5^{+}$-face. This results in $\mu'(f) = 3 - 4 + 3 \times \frac{1}{3} = 0$ by \ref{Rule1}. In the situation where a $3$-face $f$ is adjacent to one $3$-face and two $5^{+}$-faces, then we have $\mu'(f) \geq 3 - 4 + 2 \times \frac{1}{2} = 0$ or $\mu'(f) \geq 3 - 4 + 3 \times \frac{1}{3} = 0$ by \ref{Rule2}. The final charge of each $4$-face remains unchanged since it is not involved in the discharging procedure, which results in a final charge of zero. Since $G$ has no configuration as shown in \cref{fig:subfig:b-}, each $5$-face $f$ is adjacent to at most two $3$-faces, thus $f$ sends at most $\frac{1}{2}$ to each adjacent $3$-face, leading to $\mu'(f) \geq 5 - 4 - 2 \times \frac{1}{2} = 0$. 

Let $w_{1}w_{2}$ be incident with two $3$-faces $[w_{1}w_{2}w_{3}]$ and $[w_{1}w_{2}w_{4}]$, and let $w_{2}w_{3}$ be incident with a $6^{+}$-face $f = [ww_{2}w_{3}\dots w]$. Suppose that $\deg(w_{1}) = \deg(w_{2}) = 4$. Then $ww_{2}$ is incident with $f$ and another $5^{+}$-face $g$. In this case, $f$ sends $\frac{1}{2}$ to the $3$-face $[w_{1}w_{2}w_{3}]$, but no charge to $g$, we can interpret this as $f$ sending $\frac{1}{3}$ to $[w_{1}w_{2}w_{3}]$ via $w_{2}w_{3}$ and an additional $\frac{1}{6}$ to $[w_{1}w_{2}w_{3}]$ via $ww_{2}$. Hence, in any case, $f$ averagely sends at most $\frac{1}{3} = 2 \times \frac{1}{6}$ via per incident edge. As a result, each $6^{+}$-face $f$ has a final charge $\mu'(f) \geq \deg(f) - 4 - \deg(f) \times \frac{1}{3} \geq 0$. In particular, every $7^{+}$-face has a positive final charge. 

Since each $4$-vertex does not participate in the discharging procedure, its final charge remains zero. Since there are no three consecutive $3$-faces, every $k$-vertex is incident with at most $\lfloor\frac{2k}{3}\rfloor$ triangular faces. By \ref{Rule2}, each $5^{+}$-vertex $v$ sends at most $2\times \frac{1}{3} = \frac{2}{3}$ in total to two adjacent $3$-faces, but does not send anything to a singular $3$-face. Therefore, 
\begin{equation*}
\mu'(v) \geq \deg(v) - 4 - \left\lfloor \frac{\deg(v)}{3}\right\rfloor \times \frac{2}{3} > 0. 
\end{equation*}

By \eqref{Euler1}, every element in $V(G) \cup F(G)$ has a final charge of zero, which implies that $G$ is $4$-regular, $\epsilon(G) = 0$, and $G$ has only $6^{-}$-faces. If $w$ is incident with a cut-edge $ww_{1}$, then $ww_{1}$ is incident with an $8^{+}$-face, a contradiction. 

Suppose that $w$ is a cut-vertex but not incident with any cut-edge. Since $G$ is $4$-regular, $G - w$ consists of precisely two components $C_{1}$ and $C_{2}$, and $w$ has precisely two neighbors in each component. Without loss of generality, let's assume that $ww_{1}, ww_{2}, ww_{3}$ and $ww_{4}$ are the four incident edges in a cyclic order, where $w_{1}, w_{2} \in V(C_{1})$ and $w_{3}, w_{4} \in V(C_{2})$. We observe that the face $D$ incident with $ww_{2}$ and $ww_{3}$ is also incident with $ww_{1}$ and $ww_{4}$, thus $D$ must be a $6$-face with boundary $ww_{1}w_{2}ww_{3}w_{4}w$. Since $G$ is a simple $4$-regular graph, none of the edges $ww_{1}, ww_{2}, ww_{3}$ and $ww_{4}$ is incident with a $3$-face. Consequently, $D$ can only send charges through the edges $w_{1}w_{2}$ or $w_{3}w_{4}$, which implies that $\mu'(D) \geq \mu(D) - 2 \times \frac{1}{2} = 1 > 0$, leading to a contradiction. Therefore, $G$ has no cut-vertices, and is a $2$-connected $4$-regular graph. 
\end{proof}

\begin{corollary}[Cai \etal \cite{MR2682511}]
Every connected toroidal graph without $5$-cycles has minimum degree at most $3$ unless it is a $4$-regular graph. 
\end{corollary}

The following corollary is a direct consequence of \cref{MLONE}, which is stronger than that every planar graph without $5$-cycles is $3$-degenerate \cite{MR1889505}.
\begin{corollary}\label{3D}
Every planar graph without the configurations in \cref{NOA} is 3-degenerate. 
\end{corollary}
\MRONEPLANAR*
\begin{proof}
Assume that $G$ serves as a counterexample to \cref{MRONEPLANAR} with minimum number of vertices. It is observed that $G$ is connected, and $(H, f)$ is a minimal non-strictly $f$-degenerate pair. By \cref{3D}, the minimum degree of $G$ is at most $3$, contradicting \cref{L}\ref{M1}. 
\end{proof}

\begin{remark}
Note that not every toroidal graph without subgraphs isomorphic to the configurations in \cref{NOA} is $3$-degenerate. For example, the Cartesian product of an $m$-cycle and an $n$-cycle is a 2-connected 4-regular graph with Euler characteristic $\epsilon(G) = 0$. 
\end{remark}

We recall our main result regarding certain toroidal graphs. 
\MRONE*
\begin{proof}
Suppose that $G$ is a counterexample to \cref{MRONE} with minimum number of vertices. It is observed that $G$ is connected and $(H, f)$ is a minimal non-strictly $f$-degenerate pair. By \cref{MLONE}, the minimum degree of $G$ is at most $3$ or it is a 2-connected $4$-regular graph with Euler characteristic $\epsilon(G) = 0$. By \cref{L}\ref{M1}, the minimum degree of $G$ must be at least $4$, implying that $G$ is indeed a 2-connected $4$-regular graph with Euler characteristic $\epsilon(G) = 0$. Here, we have that 
\[
V(G) = \mathscr{D} = \big\{\,v \mid f(v, 1) + f(v, 2) + \dots + f(v, s) \geq \deg_{G}(v)\,\big\}. 
\]
Note that $G$ is neither a cycle nor a $4$-regular complete graph. On the other hand, $(H, f)$ is not a monoblock, which contradicts \cref{MR}. 
\end{proof}

\begin{corollary}[Cai \etal \cite{MR2682511}]
(i) Every toroidal graph without $3$-cycles is $4$-choosable. (ii) Every toroidal graph without $5$-cycles is $4$-choosable. 
\end{corollary}

Recall the structural properties of toroidal graphs without $4$-cycles. 
\MLTWO*
\begin{proof}
Suppose that $G$ is a connected toroidal graph, and it has the properties: (1) it has no $4$-cycles; (2) the minimum degree is at least $4$; and (3) it contains no $k$-cap-subgraphs for any $k \geq 5$. Since $G$ has no $4$-cycles, there are no $3$-cap-subgraphs or $4$-cap-subgraphs. We may assume that $G$ has been 2-cell embedded in the plane or torus. 

A vertex is \textbf{bad} if it is a vertex of degree $4$ and it is incident with two $3$-faces; otherwise, the vertex is \textbf{good}. A $3$-face is called a \textbf{bad $3$-face} if it is incident with a bad vertex. Let's construct $H = H(G)$, where $V(H)$ is the set of $3$-faces in $G$ incident with at least one bad vertex, and $uv \in E(H)$ if and only if the two corresponding $3$-faces in $G$ share a common bad vertex. 

We claim that: 
\begin{enumerate}
\item[($\ast$)] The graph $H(G)$ has maximum degree at most $3$. Every component of $H(G)$ is a cycle or a tree. 
\end{enumerate}
\begin{proof}[Proof of ($\ast$)]
Note that each bad vertex in $G$ corresponds to an edge in $H(G)$, thus each component of $H(G)$ contains at least one edge. Since each $3$-face is incident with at most three bad vertices, the maximum degree of $H(G)$ is at most $3$. Furthermore, as $G$ doesn't have any $k$-cap-subgraphs, there is no vertex $v$ in $H(G)$ such that it has three neighbors and is contained in a cycle. Hence, each component of $H(G)$ is a cycle or a tree. 
\end{proof}

We assign each vertex $v \in V(G)$ an initial charge $\mu(v) = \deg(v) - 6$, and similarly, each face $f \in F(G)$ an initial charge $\mu(f) = 2\deg(f) - 6$. According to Euler's formula, we have 
\begin{equation}\label{Euler}
\sum_{v\,\in\,V(G)}\big(\deg(v) - 6\big) + \sum_{f\,\in\,F(G)}\big(2\deg(f) - 6\big) = - 6\epsilon(G) \leq 0. 
\end{equation}
Next, we design a discharging procedure to get a final charge $\mu'(x)$ for each $x \in V(G) \cup F(G)$. In the context of discharging, we introduce the concept of a \textbf{bank}. For each component $D$ of $H(G)$, we allocate a bank $B(D)$, carrying an initial charge of zero. A good vertex is \textbf{incident with a bank $B(D)$} if it is incident with a bad $3$-face $f$ in $G$, and $f$ is a vertex in the component $D$ of $H(G)$. 

\begin{enumerate}[label = \textbf{R\arabic*.}, ref = R\arabic*]
\item\label{R1} Each face uniformly distributes its initial charge to all incident vertices.
\item\label{R2} Each good vertex $v$ sends $\frac{2}{5}$ to each incident bank via each incident bad $3$-face. 
\item\label{R3} For each component $D$ of $H(G)$, the corresponding bank $B(D)$ sends $\frac{2}{5}$ to each bad vertex in $G$ that corresponds to an edge in $D$. 
\end{enumerate}

A face $f$ with $\deg(f) \geq 5$ sends charge $\frac{\mu(f)}{\deg(f)} = \frac{2\deg(f) - 6}{\deg(f)} \geq \frac{4}{5}$ to each incident vertex. Since $G$ has no $4$-cycles, there are neither $4$-faces nor adjacent $3$-faces. This implies that each vertex $v$ is incident with at most $\lfloor\frac{\deg(v)}{2}\rfloor$ banks. 

Consider any vertex $v$. If $\deg(v) \geq 5$, then $v$ receives at least $\frac{4}{5} \times \lceil \frac{\deg(v)}{2} \rceil$ from its incident $5^{+}$-faces by \ref{R1}, and sends at most $\frac{2}{5} \times \lfloor\frac{\deg(v)}{2} \rfloor$ to its associated banks by \ref{R2}, which implies that $\mu'(v) \geq \deg(v) - 6 + \frac{4}{5} \times \lceil \frac{\deg(v)}{2} \rceil - \frac{2}{5} \times \lfloor\frac{\deg(v)}{2} \rfloor > 0$. If $v$ is a bad $4$-vertex, then it receives at least $\frac{4}{5}$ from each incident $5^{+}$-face by \ref{R1}, and $\frac{2}{5}$ from its associated bank by \ref{R3}, yielding $\mu'(v) \geq 4 - 6 + 2 \times \frac{4}{5} + \frac{2}{5} = 0$. If $v$ is a good $4$-vertex, then it is incident with at most one $3$-face, and then it receives $\frac{4}{5}$ from each incident $5^{+}$-face by \ref{R1}, and sends at most $\frac{2}{5}$ to its incident bank by \ref{R2}, which implies that $\mu'(v) \geq 4 - 6 + 3 \times \frac{4}{5} - \frac{2}{5} = 0$. 

Let $D$ be a component of $H(G)$ and $D$ be a cycle with $n$ vertices. Since $D$ is a cycle, each $3$-face in $G$ corresponding to a vertex in $D$ must be incident with a good vertex, and each such good vertex sends $\frac{2}{5}$ to each associated bank via each incident bad $3$-face by \ref{R2}. Thus, the final charge of the bank $B(D)$ is $\frac{2}{5}n - \frac{2}{5}n = 0$ by \ref{R2} and \ref{R3}. 

Let $D$ be a component of $H(G)$ and $D$ be a tree with $n$ vertices. For each $i$ in $\{1, 2, 3\}$, let $\alpha_{i}$ denote the number of vertices of degree $i$ in $D$. Each $3$-face in $G$ that corresponds to a vertex of degree $1$ in $D$ is incident with two good vertices, and each $3$-face in $G$ that corresponds to a vertex of degree $2$ in $D$ is incident with precisely one good vertex. Consequently, the bank $B(D)$ receives $\frac{2}{5} \times 2\alpha_{1} + \frac{2}{5} \times \alpha_{2} = \frac{4}{5} \alpha_{1} + \frac{2}{5} \alpha_{2}$ by \ref{R2}, and sends out $\frac{2}{5} |E(D)| = \frac{2}{5}(n - 1)$ by \ref{R3}. Since $n = \alpha_{1} + \alpha_{2} + \alpha_{3}$ and $2(n - 1) = \alpha_{1} + 2\alpha_{2} + 3\alpha_{3}$, we have that $2\alpha_{1} + \alpha_{2} = n + 2$. Therefore, the final charge of the bank $B(D)$ is $\frac{4}{5} \alpha_{1} + \frac{2}{5} \alpha_{2} - \frac{2}{5}(n - 1) = \frac{6}{5} > 0$. 

The discharging procedure preserves the total charge, so the sum of the final charges should be zero according to \eqref{Euler}. This implies that $G$ is $4$-regular and no component of $H(G)$ is a tree. Since $G$ has no $k$-cap-subgraphs for any $k \geq 3$, no component of $H(G)$ is a cycle. Hence, $H(G)$ does not exist and every vertex in $G$ is good. However, for every vertex $v$ with degree $4$, the final charge is $4 - 6 + 3 \times \frac{4}{5} > 0$ by \ref{R1}, a contradiction. 
\end{proof}

\MRTWO*
\begin{proof}
Suppose that $G$ is a counterexample to \cref{MRTWO} with the minimum number of vertices. It is observed that $G$ is connected and $(H, f)$ is a minimal non-strictly $f$-degenerate pair. By \cref{MLTWO,L}\ref{M1}, $G$ must contain a $k$-cap-subgraph $F$ for some $k \geq 5$. Note that 
\[
V(F) \subseteq \mathscr{D} = \big\{\,v \mid f(v, 1) + f(v, 2) + \dots + f(v, s) \geq \deg_{G}(v)\,\big\}. 
\]
Moreover, $G[V(F)]$ is $2$-connected, and it is neither a cycle nor a complete graph. On the other hand, there exists a vertex $v \in V(F)$ such that $\deg_{G[V(F)]}(v) \geq 3 > \max_{q} f(v, q)$, which contradicts \cref{L}\ref{M2}.  
\end{proof}

\section{Results on planar graphs}\label{sec:NA}
Lam \etal \cite{MR1687318} showed that every planar graph without $4$-cycles has minimum degree at most $3$ unless it contains a $5$-cap-subgraph. Borodin and Ivanova \cite{MR2586624} further improved this by showing every planar graph without triangles adjacent to $4$-cycles has minimum degree at most $3$ unless it contains a $5$-cap-subgraph. Kim and Yu \cite{MR3969022}  reiterated this structure and showed that every planar graph without triangles adjacent to $4$-cycles is DP-$4$-colorable. 

Borodin and Ivanova \cite{MR2586624} showed that every planar graph without triangles adjacent to a $4$-cycle is $4$-choosable. Additionally, Xu and Wu \cite{MR3638001} showed that a planar graph without $5$-cycles simultaneously adjacent to $3$-cycles and $4$-cycles is also $4$-choosable. Actually, they gave the following stronger structural result. 

\begin{figure}%
\centering
\subcaptionbox{\label{fig:subfig:a}}
{\begin{tikzpicture}
\coordinate (A) at (45:1);
\coordinate (B) at (135:1);
\coordinate (C) at (225:1);
\coordinate (D) at (-45:1);
\coordinate (H) at (90:1.414);
\draw (A)--(H)--(B)--(C)--(D)--cycle;
\draw (A)--(B);
        \foreach \pt in {A, B, C, D, H} {
            \node[circle, inner sep = 1.5, fill = white, draw] at (\pt) {};
        }
\end{tikzpicture}}\hspace{1.5cm}
\subcaptionbox{\label{fig:subfig:b}}
{\begin{tikzpicture}
\coordinate (A) at (45:1);
\coordinate (B) at (135:1);
\coordinate (C) at (225:1);
\coordinate (D) at (-45:1);
\coordinate (H) at (90:1.414);
\coordinate (X) at ($(A)+(0, 1)$);
\coordinate (Y) at (45:2);
\draw (A)--(Y)--(X)--(H)--(B)--(C)--(D)--cycle;
\draw (X)--(A)--(H);
        \foreach \pt in {A, B, C, D, H, X, Y} {
            \node[circle, inner sep = 1.5, fill = white, draw] at (\pt) {};
        }
\end{tikzpicture}}
\caption{Forbidden configurations in \cref{F35-Subgraph}.}
\label{A345}
\end{figure}
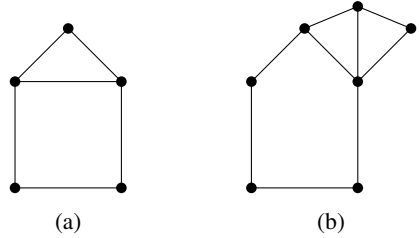

\begin{theorem}[Xu and Wu \cite{MR3638001}]\label{F35-Subgraph}
If $G$ is a planar graph without subgraphs isomorphic to the configurations as depicted in \cref{A345}, then it has minimum degree at most $3$, unless it contains a $5$-cap-subgraph.  
\end{theorem}
Here, we revisit the structural result on certain planar graphs, which refines \cref{F35-Subgraph}.
\MLTHREE*
\begin{proof}
Suppose that $G$ is a counterexample to \cref{MLTHREE}. We may assume that it is connected and it has been $2$-cell embedded in the plane. Since each $5$-cycle has no chords, each $3$-face is adjacent to at most one $3$-face and at least two $5^{+}$-faces. 

We define an initial charge function by setting $\mu(v) = \deg(v) - 4$ for every $v \in V(G)$ and $\mu(f) = \deg(f) - 4$ for every $f \in F(G)$. By Euler's  formula, the sum of the initial charges equals $-8$. That is, 
\begin{equation}\label{Euler2}
\sum_{v\,\in\,V(G)}\big(\deg(v) - 4\big) + \sum_{f\,\in\,F(G)}\big(\deg(f) - 4\big) = -8. 
\end{equation}

We define an \textbf{$(a, b)$-edge} as an edge with endpoints having degree $a$ and $b$. Let $f$ be a $5$-face incident with at least one $5^{+}$-vertex. If $f$ is adjacent to five $3$-faces, then we call $f$ a \textbf{$\bm{5_{\mathrm{a}}}$-face}. If $f$ is adjacent to precisely four $3$-faces, then we call $f$ a \textbf{$\bm{5_{\mathrm{b}}}$-face}. If $f$ is adjacent to three $3$-faces and one of the $3$-faces is adjacent to another $3$-face via a $(4, 4)$-edge, then we call $f$ a \textbf{$\bm{5_{\mathrm{c}}}$-face}. Note that the $5_{\mathrm{c}}$-face can only be as illustrated in \cref{fig:subfig:5c}. 

If $f = [w_{1}w_{2}w_{3}w_{4}w_{5}]$ is a $5$-face incident with five $4$-vertices and adjacent to a $3$-face $[ww_{1}w_{2}]$ and $w$ is a $5^{+}$-vertex, then we call $w$ a \textbf{related source} of $f$ and $f$ a \textbf{sink} of $w$. 

\begin{figure}%
\centering
\subcaptionbox{$5_{\mathrm{a}}$-face\label{fig:subfig:5a}}{\begin{tikzpicture}
\coordinate (A) at (126:1);
\coordinate (B) at (198:1);
\coordinate (C) at (270:1);
\coordinate (D) at (342:1);
\coordinate (E) at (54:1);
\coordinate (A') at (162:1.15);
\coordinate (B') at (234:1.15);
\coordinate (C') at (306:1.15);
\coordinate (D') at (18:1.15);
\coordinate (E') at (90:1.15);
\filldraw[fill=gray] (A)--(B)--(C)--(D)--(E)--cycle;
\draw (A)--(A')--(B)--(B')--(C)--(C')--(D)--(D')--(E)--(E')--cycle;
        \foreach \pt in {A, B, C, D, E, A', B', C', D', E'} {
            \node[circle, inner sep = 1.5, fill = white, draw] at (\pt) {};
        }
\end{tikzpicture}}\hspace{1.5cm}
\subcaptionbox{$5_{\mathrm{b}}$-face\label{fig:subfig:5b}}{\begin{tikzpicture}
\coordinate (A) at (126:1);
\coordinate (B) at (198:1);
\coordinate (C) at (270:1);
\coordinate (D) at (342:1);
\coordinate (E) at (54:1);
\coordinate (A') at (162:1.15);
\coordinate (B') at (234:1.15);
\coordinate (C') at (306:1.15);
\coordinate (D') at (18:1.15);
\coordinate (E') at (90:1.15);
\filldraw[fill=gray] (A)--(B)--(C)--(D)--(E)--cycle;
\draw (A)--(A')--(B)--(B')--(C)--(C')--(D)--(D')--(E);
        \foreach \pt in {A, B, C, D, E, A', B', C', D'} {
            \node[circle, inner sep = 1.5, fill = white, draw] at (\pt) {};
        }
\end{tikzpicture}}\hspace{1.5cm}
\subcaptionbox{$5_{\mathrm{c}}$-face\label{fig:subfig:5c}}{\begin{tikzpicture}
\coordinate (A) at (126:1);
\coordinate (B) at (198:1);
\coordinate (C) at (270:1);
\coordinate (D) at (342:1);
\coordinate (E) at (54:1);
\coordinate (A') at (162:1.15);
\coordinate (B') at (234:1.15);
\coordinate (C') at (306:1.15);
\coordinate (D') at (18:1.15);
\coordinate (E') at (90:1.15);
\coordinate (T) at ($(E)!(A)!(E')$);
\coordinate (Z) at ($(A)!2!(T)$); 
\filldraw[fill=gray] (A)--(B)--(C)--(D)--(E)--cycle;
\draw (B)--(B')--(C)--(C')--(D);
\draw (E) node[right]{\small$4$}--(E') node[left]{\small$4$}--(A);
\draw (E)--(Z)--(E');
        \foreach \pt in {A, B, C, D, E, B', C', E', Z} {
            \node[circle, inner sep = 1.5, fill = white, draw] at (\pt) {};
        }
\end{tikzpicture}}
\caption{An illustration of $5_{\mathrm{a}}$-face, $5_{\mathrm{b}}$-face and $5_{\mathrm{c}}$-face.}
\end{figure}
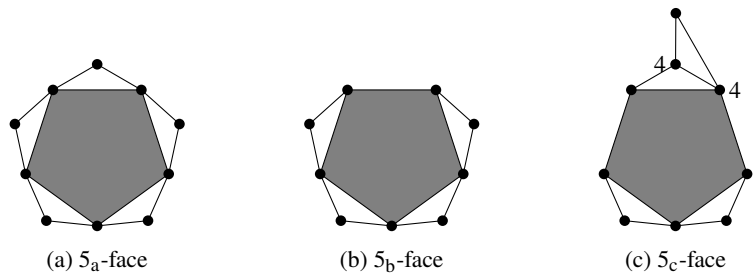

\begin{enumerate}[label = \textbf{R\arabic*.}, ref = R\arabic*]
\item\label{Rule1-} If a $3$-face is adjacent to three $5^{+}$-faces, then it receives $\frac{1}{3}$ from each adjacent face. 
\item\label{Rule2-} Let $w_{1}w_{2}$ be incident with two $3$-faces $f = [w_{1}w_{2}w_{3}]$ and $g = [w_{1}w_{2}w_{4}]$. If $\deg(w_{1}) = \deg(w_{2}) = 4$, then $f$ receives $\frac{1}{2}$ from each adjacent $5^{+}$-face; otherwise, $f$ receives $\frac{1}{3}$ from each adjacent $5^{+}$-face and $\frac{1}{3}$ from each $5^{+}$-vertex in $\{w_{1}, w_{2}\}$. 
\item\label{S} If $f$ is a $5$-face incident with five $4$-vertices and adjacent to at least four $3$-faces, then $f$ receives $\frac{1}{6}$ from each of its related sources via the adjacent $3$-face. 
\item\label{5a} If $f$ is a $5_{\mathrm{a}}$-face, then it receives $\frac{2}{3\#}$ from each incident $5^{+}$-vertex, where $\#$ is the number of incident $5^{+}$-vertices. 
\item\label{5b} If $f$ is a $5_{\mathrm{b}}$-face, then it receives $\frac{1}{3\#}$ from each incident $5^{+}$-vertex, where $\#$ is the number of incident $5^{+}$-vertices.
\item\label{5c} If $f$ is a $5_{\mathrm{c}}$-face, then it receives $\frac{1}{6}$ from each incident $5^{+}$-vertex.
\end{enumerate}

\textbf{The final charge of each face is nonnegative.} A $3$-face $f$ adjacent to three $5^{+}$-faces receives $\frac{1}{3}$ from each adjacent $5^{+}$-face, resulting in $\mu'(f) = 3 - 4 + 3 \times \frac{1}{3} = 0$ by \ref{Rule1-}. When a $3$-face $f$ is adjacent to one $3$-face and two $5^{+}$-faces, then $\mu'(f) \geq 3 - 4 + 2 \times \frac{1}{2} = 0$ or $\mu'(f) \geq 3 - 4 + 3 \times \frac{1}{3} = 0$ by \ref{Rule2-}. As $4$-faces are not involved in the discharging procedure, their final charges is zero. 

Let $f$ be a $5$-face incident with five $4$-vertices, and let $\alpha$ be the number of adjacent $3$-faces. Since $G$ has no $5$-cap-subgraphs, $f$ sends $\frac{1}{3}$ to each adjacent $3$-face by \ref{Rule1-} and \ref{Rule2-}. If $\alpha \leq 3$, then $\mu'(f) = 5 - 4 - \alpha \times \frac{1}{3} \geq 0$. If $\alpha \geq 4$, then $f$ receives $\frac{1}{6}$ from each related source, which implies that $\mu'(f) = 5 - 4 - \alpha \times \frac{1}{3} + \alpha \times \frac{1}{6} > 0$ according to \ref{Rule1-} and \ref{S}. Thus, we may assume that $f$ is incident with at least one $5^{+}$-vertex. 

If $f$ is a $5_{\mathrm{a}}$-face, then it sends $\frac{1}{3}$ to each adjacent $3$-face, yielding $\mu'(f) = 5 - 4 - 5 \times \frac{1}{3} + \# \times \frac{2}{3\#} = 0$ by \ref{Rule1-} and \ref{5a}. For a $5_{\mathrm{b}}$-face $f$, it sends $\frac{1}{3}$ to each adjacent $3$-face, resulting in $\mu'(f) = 5 - 4 - 4 \times \frac{1}{3} + \# \times \frac{1}{3\#} = 0$ by \ref{Rule1-} and \ref{5b}. A $5_{\mathrm{c}}$-face $f$ sends $\frac{1}{2}$ to an adjacent $3$-face and $\frac{1}{3}$ to each of the other adjacent $3$-faces, implying $\mu'(f) \geq 5 - 4 - \frac{1}{2} - 2 \times \frac{1}{3} + \frac{1}{6} = 0$ by \ref{Rule1-}, \ref{Rule2-} and \ref{5c}. If $f$ is incident with precisely three $3$-faces but it is not a $5_{\mathrm{c}}$-face, then it sends $\frac{1}{3}$ to each adjacent $3$-face, leading to $\mu'(f) = 5 - 4 - 3 \times \frac{1}{3} = 0$ by \ref{Rule1-} and \ref{Rule2-}. If $f$ is incident with at most two $3$-faces, then it sends at most $\frac{1}{2}$ to each adjacent $3$-face, and $\mu'(f) \geq 5 - 4 - 2 \times \frac{1}{2} = 0$ by \ref{Rule1-} and \ref{Rule2-}. 

As shown in \cref{MLONE}, every $6^{+}$-face averagely sends at most $\frac{1}{3}$ via each incident edge. Therefore, each $6^{+}$-face $f$ has a final charge $\mu'(f) \geq \deg(f) - 4 - \deg(f) \times \frac{1}{3} \geq 0$. In particular, every $7^{+}$-face has a positive final charge. 

\medskip\textbf{The final charge of each vertex is nonnegative.} Since $4$-vertices are not involved in the discharging procedure, they have a final charge of zero. Let $v$ be a $6^{+}$-vertex incident with a face $f$. If $f$ is neither a $3$-face nor a $5$-face, then $v$ sends nothing to $f$. If $f$ is a $3$-face, then $v$ may send $\frac{1}{3}$ to $f$ by \ref{Rule2-}, or $\frac{1}{6}$ via $f$ to a sink by \ref{S}. If $f$ is a $5$-face that is not a $5_{\mathrm{a}}$-face incident with precisely one $5^{+}$-vertex, then $v$ sends at most $\frac{1}{3}$ to $f$. If $f$ is a $5_{\mathrm{a}}$-face incident with precisely one $5^{+}$-vertex, then $v$ sends $\frac{2}{3}$ to $f$, but nothing to/via the two $3$-faces adjacent to $f$, otherwise there is a $5$-cap-subgraph. Hence, in all cases, $v$ sends on average at most $\frac{1}{3}$ to/via any incident face, thus $\mu'(v) \geq \deg(v) - 4 - \deg(v) \times \frac{1}{3} \geq 0$. 

Consider a $5$-vertex $v$ incident with five consecutive faces $f_{1}, f_{2}, f_{3}, f_{4}$ and $f_{5}$. We divide the discussions into four cases. 

{\bf(i)} Suppose that $v$ is incident with two adjacent $3$-faces, say $f_{1}$ and $f_{2}$. Since $5$-cycles have no chords, neither $f_{3}$ nor $f_{5}$ can be a $4^{-}$-face. If $f_{3}$ is a $5$-face, then it is adjacent to at most three $3$-faces since \cref{fig:subfig:b--} and \cref{fig:subfig:c--} are not allowed. But $f_{3}$ cannot be a $5_{\mathrm{c}}$-face, thus it receives nothing from $v$. If $f_{3}$ is a $6^{+}$-face, then it also receives nothing from $v$. In summary, $v$ sends nothing to $f_{3}$, and symmetrically sends nothing to $f_{5}$. By the discharging rules, if $f_{4}$ is a $4^{+}$-face, then $v$ sends nothing to $f_{4}$; while $f_{4}$ is a $3$-face, then $v$ may send $\frac{1}{6}$ via $f_{4}$ to a sink. Note that $v$ sends out nothing via $f_{1}$ or $f_{2}$ since there are no configurations as shown in \cref{fig:subfig:b--} and \cref{fig:subfig:c--}. Hence, $\mu'(v) \geq 5 - 4 - 2 \times \frac{1}{3} - \frac{1}{6} > 0$. 

{\bf(ii)} Suppose that $v$ is incident with two non-adjacent $3$-faces, say $f_{1}$ and $f_{3}$. If the common edge between $f_{4}$ and $f_{5}$ is a $(5, 5^{+})$-edge, then $v$ sends at most $\frac{1}{6}$ to each of $f_{4}$ and $f_{5}$ by \ref{5b} and \ref{5c}, and then $v$ sends at most $\frac{1}{3}$ in total to $f_{4}$ and $f_{5}$. If the common edge between $f_{4}$ and $f_{5}$ is a $(5, 4)$-edge, then $v$ sends $\frac{1}{3}$ to one face in $\{f_{4}, f_{5}\}$ and sends nothing to the other face, or $v$ sends at most $\frac{1}{6}$ to each of $f_{4}$ and $f_{5}$, thus $v$ sends at most $\frac{1}{3}$ in total to $f_{4}$ and $f_{5}$. In all subcases, $v$ sends at most $\frac{1}{3}$ in total to $f_{4}$ and $f_{5}$. If $v$ sends $\frac{2}{3}$ to $f_{2}$, then $f_{2}$ must be a $5_{\mathrm{a}}$-face incident with precisely one $5^{+}$-vertex and $v$ cannot be a related source of some faces, which implies $\mu'(v) \geq 5 - 4 - \frac{2}{3} - \frac{1}{3} = 0$. If $f$ sends at most $\frac{1}{3}$ to $f_{2}$, then $v$ could be related sources of two $5$-faces, which implies $\mu'(v) \geq 5 - 4 - 2 \times \frac{1}{3} - 2 \times \frac{1}{6} = 0$. 

{\bf(iii)} Suppose that $v$ is incident with precisely one $3$-face, $f_{1}$. By the discharging rules, $v$ sends at most $\frac{1}{3}$ to each of $f_{2}$ and $f_{5}$, nothing to each of $f_{3}$ and $f_{4}$, and possibly $\frac{1}{6}$ via $f_{1}$ to a sink. Hence, $\mu'(v) \geq 5 - 4 - 2 \times \frac{1}{3} - \frac{1}{6} > 0$. 

{\bf(iv)} If $v$ is not incident with any $3$-face, then $\mu'(v) = \mu(v) = 5 - 4 = 1 > 0$. 
\end{proof}

We recall our main result on certain planar graphs.
\MRTHREE*
\begin{proof}
Suppose that $G$ is a counterexample to \cref{MRTHREE} with the minimum number of vertices. It is observed that $G$ is connected and $(H, f)$ is a minimal non-strictly $f$-degenerate pair. By \cref{MLTHREE,L}\ref{M1}, $G$ must contain a $5$-cap-subgraph $F_{5}$. Note that \[
V(F_{5}) \subseteq \mathscr{D} = \big\{\,v \mid f(v, 1) + f(v, 2) + \dots + f(v, s) \geq \deg_{G}(v)\,\big\}. 
\] 
Moreover, $G[V(F_{5})]$ is $2$-connected and it is neither a cycle nor a complete graph. On the other hand, there exists a vertex $v \in V(F_{5})$ such that $\deg_{G[V(F_{5})]}(v) \geq 3 > \max_{q} f(v, q)$, which contradicts \cref{L}\ref{M2}. 
\end{proof}

\begin{remark}
All the graphs in \cref{E} contain a $3$-cycle, a $4$-cycle and a $5$-cycle, all three of which are mutually adjacent. Therefore, 
\begin{enumerate}[label = (\roman*)]
\item if $G$ is a planar graph without $3$-cycles adjacent to $4$-cycles, then it is DP-$4$-colorable and its list vertex arboricity is at most $2$; 
\item if $G$ is a planar graph without $3$-cycles adjacent to $5$-cycles, then it is DP-$4$-colorable and its list vertex arboricity is at most $2$; 
\item if $G$ is a planar graph without $4$-cycles adjacent to $5$-cycles, then it is DP-$4$-colorable and its list vertex arboricity is at most $2$. 
\end{enumerate}
\end{remark}
\begin{remark}
\cref{MLTHREE} cannot be extended to toroidal graphs; once more, the Cartesian product of an $m$-cycle and an $n$-cycle is a counterexample. Hence, it is interesting to extend \cref{MRTHREE} to toroidal graphs. 
\end{remark}

Kim and Ozeki \cite{MR3802151} pointed out that DP-coloring is also a generalization of signed (list) coloring for a signed graph $(G, \sigma)$, thus \cref{MRTHREE} implies the following conclusion, which partially extends \cite[Theorem 3.5]{MR3425977}. For comprehensive details on signed (list) coloring of signed graph, we refer the reader to \cite{MR3612419, MR3425977, MR3484719, MR3477382}. 

\begin{theorem}
If $(G, \sigma)$ be a signed planar graph and $G$ has no subgraphs isomorphic to the configurations shown in \cref{E}, then $(G, \sigma)$ is signed $4$-choosable. 
\end{theorem}

\vskip 0mm \vspace{0.3cm} \noindent\textbf{Acknowledgments.} The authors greatly appreciate the time and effort expended by the four reviewers in carefully examining our manuscript and providing valuable feedback.

\end{document}